\documentclass[11pt,a4paper]{amsart}
\usepackage[utf8]{inputenc}
\usepackage[T1]{fontenc}
\usepackage{amsmath}
\usepackage{amsfonts}
\usepackage{amssymb, lmodern}
\usepackage{amsthm}
\usepackage{accents}
\usepackage{setspace}
\usepackage{xcolor}
\usepackage{amssymb}
\usepackage{mathtools}
\usepackage{esint}
\usepackage{mathrsfs}

\definecolor{vertfonce}{rgb}{0.20, 0.46, 0.25}
\definecolor{rougefonce}{rgb}{0.64, 0.09, 0.20}
\usepackage[breaklinks=true,
            colorlinks=true,
            linkcolor=rougefonce,
            citecolor=vertfonce]{hyperref}
\usepackage{cleveref}

\definecolor{darkblue}{rgb}{0.1,0.1,0.7}
\usepackage[left=2.7cm, right=2.7cm, marginparwidth=2cm, textheight = 24cm ]{geometry}
\newcommand{\Hm}[1]{\leavevmode{\marginpar{\tiny%
			$\hbox to 0mm{\hspace*{-0.5mm}$\leftarrow$\hss}%
			\vcenter{\vrule depth 0.1mm height 0.1mm width \the\marginparwidth}%
			\hbox to
			0mm{\hss$\rightarrow$\hspace*{-0.5mm}}$\\
			\relax\raggedright #1}}}

\title[On shape optimization with large magnetic fields in two dimensions]{On shape optimization with large magnetic fields\\ in two dimensions}

\author[V.~Lotoreichik]{Vladimir Lotoreichik}
\address[V.~Lotoreichik]{Department of Theoretical Physics,
	Nuclear Physics Institute, Czech Academy of Sciences, 25068, \v{R}e\v{z}, Czech Republic}
\email{lotoreichik@ujf.cas.cz}

\author[L. Morin]{Léo Morin}
\address[L.~Morin]{Department of Mathematics, University of Copenhagen, Universitets\-parken 5, DK-2100 Copenhagen \O, Denmark}
\email{lpdm@math.ku.dk }

\newcommand\nb{\nabla}
\newcommand{\beq}{\begin{equation} \begin{split}}
\newcommand{\eeq}{\end{split} \end{equation}}

\newcommand\Omg{\Omega}

\makeatletter
\def\section{\@startsection{section}{1}\z@{.9\linespacing\@plus\linespacing}%
	{.7\linespacing} {\fontsize{13}{14}\selectfont\bfseries\centering}}
\def\paragraph{\@startsection{paragraph}{4}%
	\z@{0.3em}{-.5em}%
	{$\bullet$ \ \normalfont\itshape}}

\@addtoreset{equation}{section}
\makeatother

\renewcommand\and{\qquad\text{and}\qquad}

\newcommand\sm{\setminus}

\newcommand{\comm}[1]{}

\def\bm1{\mathbbm{1}}

\def\p{\partial}

\def\arr{\rightarrow}

\def\lm{\lambda}

\def\p{\partial}

\def\dd{{\,\mathrm{d}}}

\newcounter{counter_a}

\newcommand{\ie}{{\it i.e.}\,}
\newcommand{\cf}{{\it cf.}\,}

\numberwithin{figure}{section}
\numberwithin{equation}{section}
\theoremstyle{plain}
\newtheorem*{thm*}{Theorem}
\newtheorem{thm}{Theorem}[section]

\theoremstyle{remark}
\newtheorem{remark}[thm]{Remark}
\theoremstyle{plain}

\newcommand{\beu}{\begin{equation*}}
\newcommand{\eeu}{\end{equation*}}
\newcommand{\besu}{\begin{equation*}
\begin{aligned}}
\newcommand{\eesu}{\end{aligned}
\end{equation*}}
\newcommand{\bes}{\begin{equation}
\begin{aligned}}
\newcommand{\ees}{\end{aligned}
\end{equation}}

\newcommand\cD{\mathcal D}

\newcommand\cH{\mathcal H}

\newcommand\cL{\mathcal L}

\newcommand\ov{\overline}

\newcommand\void[1]{}

\def\ov{\overline}

   \def\sH{{\mathfrak H}}

      \def\dC{{\mathbb C}}
\def\dD{{\mathbb D}}

   \def\dN{{\mathbb N}}   
      \def\dR{{\mathbb R}}

\def\sfM{{\mathsf M}}      
      \def\sfR{{\mathsf R}}

\def\cD{{\mathcal D}}      
   \def\cH{{\mathcal H}}   
      \def\cL{{\mathcal L}}

\def\N{\mathbb{N}}

\newcommand{\dom}{\mathrm{dom}\,}

\newtheorem{theorem}{Theorem}
\newtheorem{lemma}[theorem]{Lemma}
\newtheorem{conjecture}{Conjecture}

\newcommand{\R}{\mathbb R}
\newcommand{\D}{\mathbb D}

\begin{document}

\begin{abstract}
This paper aims to show that, in the limit of strong magnetic fields, the optimal domains for eigenvalues of magnetic Laplacians tend to exhibit symmetry.
We establish several asymptotic bounds on magnetic eigenvalues to support this conclusion. Our main result implies that if, for a bounded simply-connected planar domain, the $n$-th eigenvalue of the magnetic Dirichlet Laplacian with uniform magnetic field is smaller than the corresponding eigenvalue for a disk of the same area, then the Fraenkel asymmetry of that domain tends to zero in the strong magnetic field limit. Comparable results are also derived for the magnetic Dirichlet Laplacian on rectangles, as well as the magnetic Dirac operator with infinite mass boundary conditions on smooth domains. As part of our analysis, we additionally provide a new estimate for the torsion function on rectangles.
\end{abstract}	

\maketitle
	
\section{Introduction}	
\subsection{Background and motivation}
The optimization of eigenvalues of differential operators with respect to the shape of a domain is a central theme in spectral geometry, featuring numerous intriguing results alongside challenging open questions. For the usual Laplacian with Dirichlet boundary conditions, the Faber-Krahn inequality~\cite{F23, K,K2} states that, among all bounded domains of fixed volume, the ball minimizes the lowest eigenvalue. In two dimensions, L.~Erd\H{o}s~\cite{Erdos} extended this inequality to the magnetic Laplacian with homogeneous magnetic field. A recent work~\cite{GJM} further strengthened Erd\H{o}s's result by providing a quantitative version. However, the problem of eigenvalue optimization for the magnetic Laplacian under other boundary conditions and for higher eigenvalues remains far from fully understood. Partial progress has been made in maximizing the lowest eigenvalue of the magnetic Laplacian with Neumann~\cite{CLPS24, KL24} and Robin~\cite{KL22} boundary conditions. In particular, the article~\cite{CLPS24} settles the optimization problem for the lowest eigenvalue under Neumann boundary conditions in the regime of moderate magnetic fields.

It was recently conjectured by Baur ~\cite{baur}--based on numerical evidence--that for any eigenvalue
of the magnetic Dirichlet Laplacian on bounded planar domains of fixed area, the disk is the optimal shape whenever the magnetic flux exceeds an explicit critical threshold. This observation is in shear contrast with the non-magnetic case, where the disk minimizes only the lowest eigenvalue, and is conjectured to optimize the third eigenvalue~\cite[Open problem 8]{H}. For other eigenvalues of the usual Dirichlet Laplacian on planar domains, the disk is never an optimizer, a fact which is verified analytically in~\cite{B15} by computing shape derivatives (see also the numerics from \cite{O04,AF12}).

Our aim in this note is to obtain asymptotic results supporting the conjecture of Baur. Moreover, we observe a broader pattern: optimal domains for magnetic eigenvalues tend to become symmetric as the magnetic field strength grows. We further illustrate this principle in the contexts of the magnetic Laplacian on rectangles and magnetic Dirac operators with infinite mass boundary conditions.

\subsection{Statement of the results}
Let $\Omega \subset \R^2$ be a bounded simply-connected planar domain. At this stage no regularity of the boundary of $\Omg$ is assumed. We consider the shifted magnetic Laplacian with Dirichlet boundary conditions\footnote{This operator is also one of the components of the magnetic Dirichlet-Pauli operator (see \cite{MagneticPauli}).}
\[ \mathscr{L}_{\Omega,B}u := -(\nabla -i \mathbf A)^2u -Bu,
\qquad \dom\mathscr{L}_{\Omega,B} := \big\{u\in H^1_0(\Omg)\colon\Delta u \in L^2(\Omg)\big\}.
\]
Here the vector field $\mathbf A = (A_1,A_2)^\top$ generates a homogeneous magnetic field of strength $B\geq 0$, i.e. $\partial_1 A_2 - \partial_2 A_1 = B$
and for the sake of definiteness one can use the Landau gauge $\mathbf A = (-Bx_2,0)^\top$ in the definition of the operator $\mathscr{L}_{\Omega,B}$. Since the domain $\Omg$ is simply-connected, the spectrum of $\mathscr{L}_{\Omega,B}$ does not depend on the gauge of the vector potential, at least for regular enough vector potential and  boundary of $\Omg$. In the proofs, it is more convenient to use the gauge associated with the torsion function rather than the Landau gauge. The torsion function–based gauge will be introduced later. For a comprehensive treatment of the spectral theory of the magnetic Laplacian, the reader is referred to the monographs~\cite{FH-b, R17}.

The operator $\mathscr{L}_{\Omega,B}$ is self-adjoint with  compact resolvent, and it has a non-decreasing sequence of eigenvalues,
\[ 0<\lambda_1(\Omega,B) \leq \lambda_2(\Omega,B) \leq \cdots, 
\]
which are repeated with multiplicities. It is well-known that the first eigenvalue is positive: it follows from \cite[Theorem 3.1]{HelfferSundqvist} for instance, see also Lemma \ref{lem.lowerbound} below. We are interested in the dependence on $\Omega$ of these eigenvalues when $B$ is large. Our work is in fact motivated by the following conjecture from the recent article \cite{baur}, which is supported by numerics.

\begin{conjecture}{\rm (}\cite[Conjecture 4.2]{baur}{\rm )}\label{conjecture}
Let $\Omega \subset \R^2$ be a bounded simply-connected planar domain with area $|\Omega|=1$, and $n \in \mathbb N$. Then, for all $B\ge 2\pi n$,
\begin{equation}\label{eq.conjecture}
\lambda_n(\Omega,B) \geq \lambda_n(\D,B),
\end{equation}
where $\D$ is the disk of unit area.
\end{conjecture}
We remark that the assumption of unit area can be made without loss of generality, since the general case follows by scaling.
For domains of general area, the condition on $B$ should be replaced by the respective condition on the flux, $B |\Omega| \geq 2\pi n$. As explained in \cite{baur}, the above conjecture illustrates the strong influence of magnetic fields on shape optimization problems. Indeed, when $B=0$, the minimizing shape of the Laplace eigenvalue $\lambda_n(\Omega,0)$ is typically not a disk except for $n=1$ and conjecturally for $n=3$. In fact, the numerics from \cite{baur} suggest that the lower bound $B \geq 2\pi n$ is sharp. 

As an attempt of getting more insights on this conjecture, we explain in the present article how to get asymptotic results when $B$ is large. More precisely, we show that any minimizing shape of $\lambda_n(\Omega,B)$ must converge to a disk as $B \to \infty$. This is measured in terms of the Fraenkel asymmetry,
\begin{equation}\label{def.asymmetry}
\alpha(\Omega) = \inf_{x \in \R^2} \frac{|\Omega \Delta (\cD + x)|}{|\Omg|},
\end{equation}
where $\cD$ is the disk of the same area as $\Omg$, and $\Delta$ is the symmetric difference of sets \ie $A\Delta B := (A\sm B)\cup (B\sm A)$ for $A,B\subset\dR^2$.
 The Fraenkel asymmetry is widely used in quantitative spectral isoperimetric inequalities; see the review~\cite{BP17} and the references therein. We prove the following lower bound on the magnetic eigenvalues.

\begin{theorem}\label{thm.1}
For all $n \in \mathbb N$, there exist $c_n,B_n >0$ such that the following holds. If $\Omega \subset \R^2$ is a bounded simply-connected planar domain of unit area, then for all $B \geq B_n$,
\[ \frac{\lambda_n(\Omega,B)}{\lambda_n(\D,B)} \geq 1 + c_n B \alpha(\Omega)^3 - \frac{\ln B}{c_n},\]
where $\mathbb D$ is the disk of unit area.
\end{theorem}

We emphasize that this theorem does not imply Conjecture \ref{conjecture}, even for large values of $B$. However, it supports the conjecture, by showing that optimal domains should asymptotically converge to a disk. Indeed, if $\Omega_B\subset\dR^2$ is a unit area bounded simply-connected domain such that $\lambda_n(\Omega_B,B) \leq \lambda_n(\D,B)$ then it should have small asymmetry,
\[ \alpha(\Omega_B)^3 \leq \frac{\ln B}{c_n^2 B},\]
for all $B$ sufficiently large.

In fact, our arguments also work for other shape optimization problems with strong magnetic field. For instance, we have a similar result on the optimization of $\lambda_n(\Omega,B)$ among rectangles, thus showing that an optimizer asymptotically converges to a square in the strong magnetic field limit.

\begin{theorem}\label{thm.2}
For all $n \in \mathbb N$, there exist $c_n,B_n>0$ such that the following holds. If $a>0$ and $\mathsf R_a$ is the rectangle of side lengths $a$ and $a^{-1}$, then for all $B \geq B_n$,
\[ \frac{\lambda_n(\mathsf R_a,B)}{\lambda_n(\mathsf R_1,B)} \geq 1 + c_n B \frac{(a^2-1)^2}{1+a^4} - \frac{\ln B}{c_n}. \]
\end{theorem}

The question of optimality of the square for the first magnetic eigenvalue was recently raised in \cite{K25}. Theorem \ref{thm.2} supports the conjecture from \cite{K25} that the square should always be optimal. However Theorem \ref{thm.2} only gives asymptotic convergence of the optimal rectangle to a square. Indeed, if $\mathsf{R}_a$ is such that $\lambda_n(\mathsf R_a,B) \leq \lambda_n(\mathsf R_1,B)$ then we deduce that
\[ \frac{(a^2-1)^2}{1+a^4} \leq \frac{\ln B}{c_n^2 B}, \]
for $B \geq B_n$. 

We remark that asymptotic optimality of the square was also shown for the principal eigenvalue of the Dirac operator on a rectangle with infinite mass boundary conditions in the large mass limit~\cite{BK22}. The spectral optimization for rectangles is also studied for the 
biharmonic operator~\cite{BF20}, where variables cannot be separated.

\medskip
Finally, we also get similar estimates for magnetic Dirac operators, thus suggesting that Conjecture \ref{conjecture} could have an analogue for such operators.
The non-magnetic Dirac operator with infinite mass boundary conditions
on $\Omega$ was introduced in~\cite{BFSV17a},
where self-adjointness of this operator was established. The introduction of this operator was partly motivated by the study of graphene quantum dots. 
Several later
works~\cite{CL20, CL24, LO18, PV21} were concerned with defining this operator on certain classes of non-smooth domains allowing
for corners.
In the non-magnetic setting, geometric bounds
and optimization of the smallest positive eigenvalue of the Dirac operator
with infinite mass boundary conditions on Euclidean domains were studied in~\cite{ABLO21,ABK24, BFSV17b,  BK22, DMS25, LO19}.
A counterpart of the Faber-Krahn inequality in the case of Dirac
operators on Euclidean domains remains an open problem, but
there is a numerical evidence in~\cite{ABLO21} for the validity
of this inequality, see also more recent numerical
study for higher eigenvalues in~\cite{ABK24}.  

In the magnetic case, a semi-classical analysis of the operator was performed in~\cite{BLRS23, LRR24}. We adopt some notation used therein.
In order to introduce the Dirac operator, we recall the definition of
Cauchy-Riemann operators
\begin{equation}\label{def.dz}
	\p_{\ov{z}} = \frac{\p_1+i\p_2}{2}\qquad\text{and}\qquad
	\p_z = \frac{\p_1-i\p_2}{2}
\end{equation}
and define the first-order operators
\begin{equation}\label{def.dA}
	d_{\bf A} := -2i\p_z -  A_1+i  A_2,\qquad
	d_{\bf A}^\times := -2i\p_{\ov{z}} -  A_1-i A_2,
\end{equation}
where as before $A_1$ and $A_2$ are the components of the magnetic vector potential $\mathbf A$. Let $\Omg\subset\dR^2$ be a bounded, simply-connected $\mathscr C^2$-smooth domain. We denote by $\nu = (\nu_1,\nu_2)$ the outer unit normal vector to the boundary of $\Omg$.
We remark that $\mathscr C^2$-smoothness of the domain  simplifies the definition of the Dirac operator
and allows to use methods, which are only available for more regular domains. In view of~\cite[Theorem 1.1]{BFSV17a} 
the magnetic Dirac operator
\begin{equation}
\begin{aligned}
	\mathscr{D}_{\Omega,B} u &:=
	\begin{pmatrix}
	0 & d_{\bf A} \\ d_{\bf A}^\times & 0
	\end{pmatrix}
	u,\\
	\dom\mathscr{D}_{\Omg,B} &:= \big\{u = (u_1,u_2)\in H^1(\Omg;\dC^2)\colon
u_2 = i(\nu_1+i\nu_2) u_1~ {\rm{on}}~\partial \Omega\big\}
\end{aligned}
\end{equation}
is self-adjoint on the Hilbert space $L^2(\Omg;\dC^2)$. 
Here, we implicitly used that adding the magnetic field leads to a bounded symmetric additive perturbation of the non-magnetic Dirac operator with infinite mass boundary conditions. Thus, adding the magnetic field does not influence self-adjointness.
The spectrum of this Dirac operator is purely discrete thanks to compact embedding of $H^1(\Omega)$ into $L^2(\Omega)$.
By~\cite[Proposition 1.5]{BLRS23}
we have $0\notin\sigma(\mathscr{D}_{\Omg,B})$. It is clear that the magnetic field manifests in the spectrum of $\mathscr{D}_{\Omg,B}$ and we  denote by $\lm_n^+(\Omg,B)$ (respectively by $\lm_n^-(\Omg,B)$) the positive (respectively, negative) eigenvalues of $\mathscr D_{\Omg,B}$ enumerated in non-decreasing (respectively, non-increasing) way and counted with multiplicities,
\[
	\dots\le\lm_2^-(\Omg,B)\le \lm_1^-(\Omg,B)< 0 < \lm_1^+(\Omg,B)\le \lm_2^+(\Omg,B)\le \dots
\]
We prove the following bound on the positive eigenvalues of $\mathscr{D}_{\Omega,B}$.

\begin{theorem}\label{thm.3}
For all $n \in \mathbb N$, there exist $c_n,B_n >0$ such that the following holds. If $\Omega \subset \R^2$ is a bounded, simply-connected, $\mathscr{C}^2$-smooth domain with unit area, then for all $B \geq B_n$,
\[ \frac{\lambda_n^+(\Omega,B)}{\lambda_n^+(\mathbb D,B)} \geq 1 + c_n B \alpha(\Omega)^3 - \frac{\ln B}{c_n}.\]
\end{theorem}

The proofs of Theorems \ref{thm.1}, \ref{thm.2} and \ref{thm.3} follow the same strategy, and rely on three ingredients:
\begin{itemize}
\item Upper bounds on the $n$-th eigenvalue of the optimizing shape (i.e. on $\lambda_n(\mathbb D,B)$, $\lambda_n(\mathsf R_1,B)$ and $\lambda_n^+(\mathbb D,B)$) that are asymptotically sharp when $B$ is large. These bounds are given in Lemma \ref{lem.disk} and \ref{lem.upper.dirac}, respectively, and come from \cite{MagneticPauli,BLRS23}. The bound in Lemma~\ref{lem.disk} is stated for general convex domains and applies to both $\dD$ and $\sfR_1$.
\item Lower bounds on $\lambda_1(\Omega,B)$ and $\lambda_1^+(\Omega,B)$ that are uniform with respect to $\Omega$, and still catch the exponentially small behaviour as $B$ increases. For the magnetic Laplacian, this bound comes from \cite{HelfferSundqvist} (see Lemma \ref{lem.lowerbound}). In the case of the magnetic Dirac operator, we adapt the argument to prove Lemma \ref{lem.lower.dirac}.
\item In these asymptotic estimates, the torsion function $\varphi^\Omega$ (defined in Section \ref{sec.torsion}) plays a crucial role. In fact, the smallness of the eigenvalues is measured in terms of the maximum of the torsion function. It follows from Talenti's comparison principle that $\| \varphi^\Omega \|_\infty$ is maximized by the disk among domains of fixed area. We need a quantitative version of this result, see Lemma \ref{lem.torsion}, which can also be derived as a consequence from more general recent quantitative Talenti's comparison principle~\cite{AL25,ABMP23}.
In the case of rectangles, we prove that the square maximizes $\| \varphi^\Omega \|_\infty$ in Theorem \ref{thm.rect}, also in quantitative form. The latter result on the torsion function is new to the best of our knowledge.
\end{itemize}
It remains to outline the structure of the paper.
Section~\ref{sec.torsion} is devoted to the torsion function
and to quantitative isoperimetric inequalities for its maximal value.
We prove Theorems \ref{thm.1} and \ref{thm.2} in Section~\ref{sec.magnetic.laplacian}  while the proof of Theorem~\ref{thm.3} is given in Section~\ref{sec.dirac}. The paper is complemented by an appendix, where we provide a new proof of
an essentially known result on the torsion function, which is central for our analysis.

\section{On the maximum of the torsion function}\label{sec.torsion}

The spectra of the magnetic Dirichlet Laplacian $\mathscr{L}_{\Omega,B}$ and of the magnetic Dirac operator $\mathscr{D}_{\Omega,B}$ in strong magnetic fields are deeply related to the \emph{torsion function} $\varphi^\Omega$. This function is defined as the solution to the elliptic equation
\begin{equation}
\begin{cases}
-\Delta \varphi^\Omega = 1 &\quad {\rm in} \quad \Omega,\\
\varphi^\Omega = 0 &\quad {\rm on} \quad \partial \Omega.
\end{cases}
\end{equation}
Note that this elliptic equation has a unique solution in the Sobolev space $H^1_0(\Omg)$.
By the maximum principle, the function $\varphi^\Omega$ is positive inside $\Omega$, and we denote by $\varphi^\Omega_{\rm m}$ its maximum.

It is standard that the maximum of the torsion function is always maximized by the disk among domains of fixed area. This follows from Talenti's comparison principle \cite[Theorem 1]{Talenti}. In fact, we also have a quantitative version of this result, which takes into account the Fraenkel asymmetry $\alpha(\Omega)$ defined in \eqref{def.asymmetry}.

\begin{lemma}\label{lem.torsion} There is a $c >0$ such that, for all bounded open sets $\Omega \subset \R^2$,
\[ \varphi_{\rm m}^{\cD} - \varphi_{\rm m}^\Omega \geq c |\Omega| \alpha(\Omega)^3\]
where $\cD\subset\dR^2$ is a disk of the same area as $\Omega$.
\end{lemma}
\begin{remark}
	Our analysis combined with a bound in~\cite{LZDJ15} on the best constant in the quantitative isoperimetric inequality~\cite{FMP} used in the proof of the above lemma yields as a by-product that we can choose 
	$c = \frac{1}{1024\pi}$.  
\end{remark}
The cubic-power term $\alpha(\Omega)^3$ on the right-hand side could probably be improved to $\alpha(\Omega)^2$, but we do not need such a precise result for our purposes. In fact, this improvement is considered to be an important open problem; see the discussion in~\cite{AL25}. Lemma~\ref{lem.torsion} was first obtained in \cite{Kim}, using probabilistic methods. For completeness we suggest the proof outsourced to the appendix and inspired by \cite[Theorem 1.2]{ABMP23}. In fact, this result in \cite{ABMP23} uses ideas from \cite{HN}, which were also adapted to the magnetic setting in \cite{GJM}.

Lemma~\ref{lem.torsion} has a counterpart for rectangular domains. The main additional ingredient in the case of rectangles is the fact that $\varphi^{\mathsf R_a}_{\rm m}$ is maximized by the square ($a=1$). We prove the following
stronger quantitative estimate on the maximum of the torsion function on rectangles $\mathsf R_a=(-a/2,a/2)\times (- \frac{1}{2a},\frac{1}{2a})$. This property of the torsion function on rectangles is new to the best of our knowledge. 
\begin{theorem}\label{thm.rect}
	For all $a>0$ we have
	\[  \varphi_{\rm m}^{\mathsf R_1} - \varphi_{\rm m}^{\mathsf R_a} \geq \frac{(a^2-1)^2}{24(1+a^4)}.\]
\end{theorem}

\begin{proof}
	The torsion function of the rectangle $\mathsf R_a$ is explicitly given by
	\[ \varphi^{\mathsf R_a}(x_1,x_2) = \frac{16}{\pi^4} \sum_{n, m \geq 0} \frac{\sin \Big( \frac{(2n+1)\pi}{2a}(2x_1+a) \Big) \sin \Big( \frac{(2m+1)\pi a}{2}(2x_2+ \frac 1a) \Big)}{(2n+1)^3(2m+1)} \frac{a^2}{1+ a^4\frac{(2m+1)^2}{(2n+1)^2}},\]
	as one can find using Fourier series; \cf~\cite[Section 7]{BM24}. Moreover, its maximum is reached at the center $x=0$, and thus
	\[ \varphi_{\rm m}^{\mathsf R_a} = \frac{16}{\pi^4} \sum_{n,m \geq 0} \frac{(-1)^{n+m}}{(2n+1)^2(2m+1)^2}f \Big( a^2 \frac{2m+1}{2n+1} \Big),\]
	where $f(x) = \frac{x}{1+x^2}$. We extract from the sum the diagonal terms $m=n$,
	\[ \varphi_{\rm m}^{\mathsf R_a} = \frac{16}{\pi^4} \sum_{n \geq 0} \frac{1}{(2n+1)^4} f(a^2) + \frac{16}{\pi^4} \sum_{m >n}  \frac{(-1)^{n+m}}{(2n+1)^2(2m+1)^2} \Big( f \Big( a^2 \frac{2m+1}{2n+1} \Big) + f \Big( a^2 \frac{2n+1}{2m+1} \Big) \Big).\]
	Therefore, we have to get a lower bound for
	\[ \varphi_{{\rm m}}^{\mathsf R_1} - \varphi_{{\rm m}}^{\mathsf R_a} = \frac{16}{\pi^4} \sum_{n \geq 0} \frac{1}{(2n+1)^4} \big( f(1)-f(a^2) \big) + \frac{16}{\pi^4} \sum_{m >n}  \frac{(-1)^{n+m}}{(2n+1)^2(2m+1)^2} g \Big(a^2, \frac{2m+1}{2n+1} \Big),\]
	where $g(x,y) = 2f(y) - f(xy) - f(x/y)$.  By Lemma \ref{lem.g} below, the rational function $g$ is bounded by $|g(x,y)| \leq g(x,1) = 2(f(1)-f(x))$. Thus,
	\[ \varphi_{{\rm m}}^{\mathsf R_1} - \varphi_{{\rm m}}^{\mathsf R_a} \geq \frac{16}{\pi^4} \Big( \sum_{n \geq 0} \frac{1}{(2n+1)^4} - \sum_{m \neq n}  \frac{1}{(2n+1)^2(2m+1)^2} \Big) \big( f(1)-f(a^2) \big).\]
	We recall that
	\[ \sum_{n\geq 0} \frac{1}{(2n+1)^4} = \frac{15}{16} \zeta(4) = \frac{\pi^4}{96}, \quad \sum_{m \neq n}  \frac{1}{(2n+1)^2(2m+1)^2} = \frac{\pi^4}{192}.\]
	We deduce that
	\[ \varphi_{{\rm m}}^{\mathsf R_1} - \varphi_{{\rm m}}^{\mathsf R_a} \geq \frac {1}{12} (f(1)-f(a^2)) = \frac{(1-a^2)^2}{24(1+a^4)}.\qedhere\]
\end{proof}
To prove Theorem \ref{thm.rect}, we used a technical bound on a specific rational function.

\begin{lemma}\label{lem.g}
	Define the rational function $g$ by
	\[g(x,y) = \frac{2y}{1+y^2} - \frac{xy}{1+x^2y^2} - \frac{xy}{x^2+y^2}. \]
	Then, for all $x,y>0$,
	\[|g(x,y)| \leq g(x,1).\]
\end{lemma}

\begin{proof}
	First note that we can assume $x,y \geq 1$ because of the symmetry $g(x,y) = g(x^{-1},y) = g(x,y^{-1})$. The function $g$ can be factorized as
	\begin{equation}\label{eqg}
		g(x,y) = \frac{y(1-x)^2(2x^2y^2 + 2y^2 - x(y^2-1)^2)}{(1+y^2)(1+x^2y^2)(x^2+y^2)}. 
	\end{equation}
	When $g(x,y) >0$, we simply remove the negative terms,
	\[g(x,y) \leq \frac{2y^3(1-x)^2(x^2+1)}{(1+y^2)(1+x^2y^2)(x^2+y^2)}. \]
	Since $2y \leq 1+y^2$, we have
	\[g(x,y) \leq \frac{y^2(1-x)^2(x^2+1)}{(1+x^2y^2)(x^2+y^2)}.\]
	The function $F_x : y \mapsto \frac{y^2}{(1+x^2y^2)(x^2+y^2)}$ is decreasing on $[1,\infty)$, as one can check by computing the derivative $(\ln F_x)'(y)$. Therefore, $F_x(y) \leq F_x(1)$ and thus
	\[ g(x,y) \leq \frac{(1-x)^2}{1+x^2} = g(x,1).\]
We now have to bound $|g|$ when $g(x,y) <0$. Removing the negative terms again we find
\[ -g(x,y) \leq \frac{y(1-x)^2x(y^2-1)^2}{(1+y^2)(1+x^2y^2)(x^2+y^2)}.\]
We use $2xy \leq x^2 + y^2$, and remove the $1$'s in the denominator to find
\[ -g(x,y) \leq \frac{(1-x)^2(y^2-1)^2}{2x^2 y^4} \leq \frac{(1-x)^2}{2x^2}. \]
Finally $2x^2 \geq 1+x^2$, and thus
\[ -g(x,y) \leq \frac{(1-x)^2}{1+x^2} = g(x,1).
\qedhere
\]
\end{proof}

\section{Eigenvalues of the magnetic Laplacian}\label{sec.magnetic.laplacian}

In the limit of large magnetic fields, it is known that the eigenfunctions of the magnetic Laplacian $\mathscr{L}_{\Omega,B}$ are exponentially localized near the maximum points of the torsion function $\varphi^\Omega$. In fact, this was one of the crucial observations in \cite{MagneticPauli}, where asymptotics of the form
\[ C_n' B^{n+1} e^{-2B \varphi_{\rm m}^\Omega} \leq \lambda_n(\Omega,B) \leq C_n'' B^{n+1} e^{-2B \varphi_{\rm m}^\Omega}, \quad \text{for $B$ large enough,}\] were proven\footnote{In fact, there are some assumptions on $\Omega$ in \cite{MagneticPauli}, including smoothness, and uniqueness of the maximum of $\varphi^\Omega$. This last property is always satisfied for convex domains (see \cite[Theorem 1]{KL87}).}. This result was the origin of our ideas to prove Theorem~\ref{thm.1}. However, the bounds in \cite{MagneticPauli} are far from uniform with respect to $\Omega$, hence it is impossible to apply them directly in our setting. We shall only use the upper bound, in the specific cases of a disk or square, but we will state the bound for more general convex domains. 

\begin{lemma}\label{lem.disk}
Let $n \in \N$ and let $\Omega\subset\dR^2$ be a bounded convex domain. Then there exist $M_n  = M_n(\Omega) >0$ and $C_n = C_n(\Omega) >0$ such that
\[ \lambda_n(\Omega,B) \leq C_n B^{n+1} e^{-2B \varphi_{\rm m}^\Omega},\]
for all $B \geq M_n$.
\end{lemma}

The proof of the above lemma is essentially given in \cite{MagneticPauli} and even with a sharp asymptotic constant $C_n$ together with the corresponding lower bound. We present here a simpler version of the argument from \cite{MagneticPauli}, in which we do not aim at a sharp constant. Thanks to this simplification we could reduce the regularity assumptions on the boundary of the domain compared to~\cite{MagneticPauli}. We also point out that convexity of $\Omega$ is only used to guarantee that the torsion function of $\Omega$ has a unique global maximum with negative definite Hessian at the maximum point. We can instead assume that the domain $\Omega$ is such that the desired property of the torsion function holds, enlarging thus the class of domains to which the lemma applies. This extension is not needed for our purposes since we intend to apply this lemma only to disks and squares (both being convex domains).

\begin{proof}[Proof of Lemma~\ref{lem.disk}]
Recall that, for any convex domain, the square root of the torsion function is strictly concave~\cite[Theorem 1]{KL87} and
thus the torsion function has a unique global maximum.
Moreover, the Hessian of the torsion function at the point of the maximum
is negatively definite~\cite[Theorem 1]{S18}.

Without loss of generality we can assume that the domain $\Omg$ is positioned so that the unique  maximum of its torsion function is attained at the origin.
Note that since $\Omega$ is convex and, in particular, simply connected, we can choose the specific gauge of the magnetic field
\begin{equation}
A_1 = B \partial_2 \varphi^\Omg, \qquad A_2 = - B \partial_1 \varphi^\Omega, 
\end{equation}
which indeed satisfies $\partial_1 A_2 - \partial_2 A_1 = B$. Then, with the notations \eqref{def.dz} and \eqref{def.dA} we formally have
\[ \mathscr{L}_{\Omega, B} = -(\nabla - i\mathbf A)^2 - B = d_{\mathbf A}^{} d^\times_{\mathbf A}.\]
In particular, the quadratic form associated to $\mathscr{L}_{\Omega,B}$ can be written as
\[ \mathscr{Q}[\Psi] := \int_\Omg | d_{\mathbf A}^\times \Psi |^2 \dd x, \qquad \dom\mathscr{Q} :=  H^1_0(\Omg). \]
Now, writing $v = e^{-B \varphi^\Omg} \Psi$ for $\Psi\in H^1_0(\Omg)$, we have $v\in H^1_0(\Omg)$ and $d_{\mathbf A}^\times \Psi = - 2 i e^{B \varphi^\Omg} \partial_{\bar z} v$. This is due to our specific choice of gauge. In particular, by the min-max principle we have
\begin{equation}\label{eq.minmax}
\lambda_n(\Omg,B) = \inf_{\substack{V \subset H^1_0(\Omg) \\ \dim V =n}} \sup_{v \in V\setminus \{0\}} \frac{\displaystyle 4 \int_\Omg | \partial_{\bar z} v |^2 e^{2B \varphi^\Omg} \dd x }{\displaystyle\int_\Omg |v|^2 e^{2B \varphi^\Omg} \dd x}.
\end{equation}
We consider the following space of trial functions, $V= {\mathrm{span}} \{v_0, \dots, v_{n-1}\}$ with
\[ v_k(x) = \chi_B(x) B^{(k+1)/2} (x_1+ix_2)^k e^{- B \varphi^{\Omg}_{\rm m}}, \]
where the cut-off function $\chi_B\colon\Omg\arr\dR$ is defined as
\begin{equation}\label{eq:chiB}
	\chi_B(x) := \chi\big(B\cdot{\rm dist}\,(x,\p\Omg)\big)
\end{equation} 
with some $\chi \in C^\infty(\ov\dR_+)$ satisfying $0\le \chi\le 1$,
$\chi(0) = 0$ and $\chi(t) = 1$ for all $t \ge 1$. Here ${\rm dist}\,(\cdot,\p\Omg)$ stands for the usual distance function to the boundary of $\Omg$. This distance function is known to be $1$-Lipschitz~\cite[Section 3]{DZ94}. Hence, we infer that $\chi_B$ is Lipschitz continuous up to the boundary of $\Omg$.

When $v = \sum_{k=0}^{n-1} \alpha_k v_k$ with some complex coefficients $\{\alpha_k\}_{k=0}^{n-1}$, we can estimate its norm,
\begin{align*}
\int_\Omg |v|^2 e^{2B \varphi^\Omg} \dd x &= 
\int_\Omg \chi_B(x)^2 \Big| \sum_{k=0}^{n-1} \alpha_k B^{k/2}(x_1+ix_2)^k \Big|^2 e^{-2B( \varphi^\Omg_{\rm m} - \varphi^\Omg)}B  \dd x, \\
& \geq \int_{\cD(0,B^{- \frac 12})} \Big| \sum_{k=0}^{n-1} \alpha_k B^{k/2}(x_1+ix_2)^k \Big|^2 e^{-\varepsilon B |x|^2}B \dd x,
\end{align*}
where we restricted the integral to a small disk centred at the origin  $0$ with the radius $B^{-\frac12}$ (provided that $B$ is large enough), and used the fact that the Hessian of $\varphi^\Omg$ at the point of its maximum (the origin) is negative definite; here $\varepsilon$ is some positive constant whose existence is guaranteed, but whose precise admissible value is not essential for what follows. Then making a change of variable $y = \sqrt{B} x$
we get
\begin{align*}
\int_\Omg |v|^2 e^{2B \varphi^\Omg} \dd x \geq \int_{\cD(0,1)} \Big|\sum_{k=0}^{n-1} \alpha_k (y_1+iy_2)^k \Big|^2 e^{-\varepsilon|y|^2} \dd y.
\end{align*}
Note that the quantity 
\[N(P):= \left( \int_{\cD(0,1)} | P(y_1 + i y_2) |^2 e^{-\varepsilon |y|^2} \dd y \right)^{\frac 12}\] is a norm on the space of polynomial of degree $\leq n-1$. Since this is a finite dimensional space, this norm is equivalent to the Euclidean norm of the coefficients: there is a constant $K_{n}>0$ such that
\[\int_{\cD(0,1)} \Big|\sum_{k=0}^{n-1} \alpha_k (y_1+iy_2)^k \Big|^2 e^{-\varepsilon |y|^2} \dd y \geq K_{n}  \sum_{k=0}^{n-1}|\alpha_k|^2,\]
and $K_{n}$ is independent of the coefficients $\alpha_k$. We deduce that
\begin{equation}\label{est.norm}
\int_\Omg |v|^2 e^{2B \varphi^\Omg} \dd x \geq K_n \sum_{k=0}^{n-1}|\alpha_k|^2.
\end{equation}
In particular, we have shown as a by-product that $\dim V = n$.

We can also estimate the numerator in the min-max characterisation of $\lm_{n}(\Omg,B)$ evaluated on the trial function of $v_k$. Since $\partial_{\bar z} (x_1+ix_2)^k=0$, we have
\begin{align*}
4 \int_\Omg  | \partial_{\bar z} v_k |^2 e^{2B \varphi^\Omg} \dd x =B^{k+1}e^{-2 B  \varphi_{\rm m}^\Omg}\int_\Omg  | \nabla \chi_B |^2 e^{2B \varphi^\Omg} \dd x .
\end{align*}
The gradient of $\chi_B$ is supported on $B^{-1}$-neighbourhood of $\p\Omg$ \[
\hat\Omg_{B^{-1}}:=\{x\in\Omg\colon{\rm dist}\,(x,\p\Omg)\le B^{-1}\},\] 
and there exists a constant $C >0$  such that $|\nb \chi_B| \le CB$ for all sufficiently large $B >0$. Moreover, since $\varphi^\Omega$ vanishes on the boundary, we have $\varphi^\Omg \leq C' B^{-1}$ on $\hat\Omg_{B^{-1}}$ with some constant $C' > 0$ (It follows from boundedness of $\nabla \varphi^\Omg$, see \cite[Eq. 6.12]{S81} or \cite{HS}). Thus, we can find a constant $L_k >0$ (independent of $B$) such that
\begin{equation}\label{est.energy0}
4 \int_\Omg  | \partial_{\bar z} v_k |^2 e^{2B \varphi^\Omg} \dd x \leq L_k B^{k+3} e^{-2B \varphi_{\rm m}^\Omg} | \hat\Omega_{B^{-1}}|.
\end{equation}
Note that $|\hat \Omega_{B^{-1}}| \leq C'' B^{-1}$ since
any convex domain has Lipschitz boundary.  Coming back to \eqref{est.energy0}, we obtain
\begin{equation}\label{est.energy}
4 \int_\Omg  | \partial_{\bar z} v_k |^2 e^{2B \varphi^\Omg} \dd x \leq C'' L_k B^{k+2} e^{-2B \varphi_{{\rm m}}^\Omega}.
\end{equation}
When $v = \sum_{k=0}^{n-1} \alpha_k v_k$ is an arbitrary element of $V\setminus\{0\}$, we deduce from \eqref{est.energy} and the triangle inequality that
\[4 \int_\Omg  | \partial_{\bar z} v |^2 e^{2B \varphi^\Omg} \dd x \leq L_n' B^{n+1} e^{-2B \varphi_{\rm m}^\Omega} \sum_{k=0}^{n-1} |\alpha_k|^2,\]
with a new positive constant $L_n'$. Combining with \eqref{est.norm} we obtain, for all $v \in V\setminus\{0\}$,
\[ \frac{\displaystyle 4 \int_\Omg | \partial_{\bar z} v |^2 e^{2B \varphi^\Omg} \dd x }{\displaystyle\int_\Omg |v|^2 e^{2B \varphi^\Omg} \dd x} \leq \frac{L_n'}{K_n} B^{n+1} e^{-2B \varphi_{\rm m}^\Omega}  \]
and the estimate on $\lambda_n(\Omega,B)$ follows from the min-max principle \eqref{eq.minmax}.
\end{proof}

It is also possible to get lower bounds on the magnetic eigenvalues that are uniform with respect to $\Omega$, and that capture the exponential behaviour as $B \to \infty$. 

\begin{lemma}[{\cite[Theorem 3.1]{HelfferSundqvist}}]\label{lem.lowerbound}
Assume that $\Omega\subset\dR^2$ is a bounded simply connected domain of unit area. Then, for all $B \geq 0$,
\[ \lambda_1(\Omega,B) \geq \lambda_1(\D,0) e^{-2B \varphi_{\rm m}^\Omega}. \]
\end{lemma}

We are now in position to prove Theorem \ref{thm.1}.
\begin{proof}[Proof of Theorem~\ref{thm.1}]
Let $n\in\dN$ be fixed.	
The upper bound from Lemma~\ref{lem.disk} applied to $\D$ and the uniform lower bound from Lemma \ref{lem.lowerbound} applied to $\Omg$ give for all $B \ge M_n(\D)$
\begin{equation}
\frac{\lambda_n(\Omega,B)}{\lambda_n(\D,B)} \geq \frac{\lambda_1(\D,0) e^{2B (\varphi_{\rm m} ^\D - \varphi_{\rm m}^\Omega)}}{C_n(\D) B^{n+1}}.
\end{equation}
We then use the quantitative estimate on the torsion function (Lemma \ref{lem.torsion}) to obtain for all $B  \ge M_n(\D)$
\begin{align*}
\frac{\lambda_n(\Omega,B)}{\lambda_n(\D,B)} &\geq \frac{\lambda_1(\D,0) e^{2c B \alpha(\Omega)^3}}{C_n(\D) B^{n+1}} = \frac{\lambda_1(\D,0)}{C_n(\D)} e^{2cB \alpha(\Omega)^3 - (n+1) \ln B} \\
&\geq \frac{\lm_1(\D,0)}{C_n(\D)}\left(2 c  B \alpha(\Omega)^3 - (n+1) \ln B\right),
\end{align*}
where we used the inequality $e^x \ge x$ in the last step.
Note that there exists $B_n > M_n(\D)$ and $c_n > 0$
which only depend on the index $n$ of the eigenvalue
such that
\[
\frac{\lm_1(\D,0)}{C_n(\D)}\left(2 c  B \alpha(\Omega)^3 - (n+1) \ln B\right) \ge 1+ c_n B\alpha(\Omega)^3 -\frac{\ln B}{c_n}
\]
for all $B \ge B_n$. Thus, the claim of the theorem follows.
\end{proof}

Following the same circle of ideas as in the proof of Theorem~\ref{thm.1} we can prove Theorem~\ref{thm.2}.
\begin{proof}[Proof of Theorem~\ref{thm.2}]
Let $n\in\dN$ be fixed.
The upper bound from Lemma~\ref{lem.disk} applied to the square $\sfR_1$, and the lower bound from Lemma \ref{lem.lowerbound}
applied to the rectangle $\sfR_a$ yield for all $B \ge M_n(\sfR_1)$
\[ \frac{\lambda_n(\mathsf R_a ,B)}{\lambda_n(\mathsf R_1,B)} \geq \frac{\lambda_1(\D,0) e^{2B (\varphi_{\rm m}^{\mathsf R_1} - \varphi_{\rm m}^{\mathsf R_a})} }{C_n(\sfR_1) B^{n+1}}. \]
By Theorem \ref{thm.rect} we have
\[ \varphi_{\rm m}^{\mathsf R_1} - \varphi_{\rm m}^{\mathsf R_a} \geq \frac{(a^2-1)^2}{24(1+a^4)}. \]
We deduce that
\[ \frac{\lambda_n(\mathsf R_a ,B)}{\lambda_n(\mathsf R_1,B)} \geq   \frac{\lambda_1(\D,0)}{C_n(\sfR_1) } \left(B\frac{(a^2-1)^2}{12(1+a^4)} -(n+1) \ln B\right), \]
where we again used the inequality $e^x \ge x$.
Clearly, there exist $c_n>0$, $B_n > M_n(\mathsf R_1)$ such that
\[
  \frac{\lambda_1(\D,0)}{C_n(\sfR_1) } \left(B\frac{(a^2-1)^2}{12(1+a^4)} -(n+1) \ln B\right)
  \ge 1 + c_n B \frac{(a^2-1)^2}{1+a^4} - \frac{\ln B}{c_n}
\]
for all $B \ge B_n$ and the claim of the theorem follows.
\end{proof}

\section{Magnetic Dirac operator with infinite mass boundary conditions}\label{sec.dirac}

In this section, we consider the magnetic Dirac operator $\mathscr{D}_{\Omega,B}$ with infinite mass boundary
conditions on a bounded planar simply-connected $\mathscr{C}^2$-smooth domain $\Omega$. We prove Theorem~\ref{thm.3} on the asymptotic minimizers of the positive eigenvalues. Since the domain $\Omega$ is simply-connected, we can choose any gauge $\mathbf A$ generating the homogeneous magnetic field $B$. It is convenient to use the following gauge, defined through the torsion function,
\[ A_1 =  B \partial_2 \varphi^\Omega, \qquad A_2 = - B \partial_1 \varphi^\Omega,\]
where $\varphi^\Omg$ is as before the torsion function of $\Omg$.
 In fact, the positive eigenvalues $\mathscr{D}_{\Omg,B}$ can be efficiently characterised by a non-linear variational principle.
In order to formulate it, we need to introduce several additional function spaces. Namely, the magnetic Hardy space
\[
\mathscr{H}_{\bf A}^2(\Omg) := \big\{u\in L^2(\Omg)\colon d^\times_{\bf A} u = 0, \, u|_{\p\Omg}\in L^2(\p\Omg)\big\}
\]
and the Hilbert space
\[
	\mathfrak{H}_{\bf A}(\Omg) := H^1(\Omg) + \mathscr{H}_{\bf A}^2(\Omg)
\]
with the inner product
\[
	\langle u,v\rangle_{\mathfrak{H}_{\bf A}(\Omg)} := 
	\langle u,v\rangle_{L^2(\Omg)}+ 	\langle d_{\bf A}^\times u,d^\times_{\bf A}v\rangle_{L^2(\Omg)}+  \langle u|_{\p\Omg},v|_{\p\Omg}\rangle_{L^2(\p\Omg)}.
\]
In fact, $\mathscr{H}_{\bf A}^2(\Omg)$ with the scalar product $\langle u|_{\p\Omg},v|_{\p\Omg}\rangle_{L^2(\p\Omg)}$ is a Hilbert space (see \cite[Section 2.1]{BLRS23} for details and more properties of this space). Then we have the following non-linear min-max principle, for all $n\in\dN$,
\begin{equation}\label{eq:minmaxDirac}
	\lm_n^+(\Omg,B) = \inf_{\begin{smallmatrix} \cL\subset\sH_{{\bf A}}(\Omg)\\ \dim \cL = n\end{smallmatrix}}\sup_{u\in\cL\sm\{0\}} \frac{\|u\|^2_{L^2(\p\Omg)} + \sqrt{\|u\|^4_{L^2(\p\Omg)}+ 4\|u\|^2_{L^2(\Omg)} \|d_{\bf A}^\times u\|^2_{L^2(\Omg)}}}{2\|u\|^2_{L^2(\Omg)}}.
\end{equation}
This principle for $\lm_1^+(\Omg,0)$ appeared for the first time in~\cite[Theorem 4]{ABLO21}. The magnetic case and the higher eigenvalues are covered by a more general result in~\cite[Theorem 1.9]{BLRS23}.
\medskip

In the analysis of magnetic Dirac eigenvalues we will employ certain geometric bounds on these eigenvalues in terms of the torsion function.
The first bound in the case of the disk is a direct consequence of the general result in~\cite[Theorem 1.11]{BLRS23}.

\begin{lemma}\label{lem.upper.dirac}
	For any $n\in\dN$, there exist constants $C_n, M_n > 0$ such that
	\[
		\lm_n^+(\D,B) \le C_n B^n e^{-2B \varphi^{\mathbb D}_{\rm m}}
	\] 
	holds for all $B \ge M_n$.
\end{lemma}
The second estimate is a lower bound for general domains and its proof relies on an adaptation of the trick in~\cite{HelfferSundqvist} (Lemma \ref{lem.lowerbound}) to the non-linear variational characterisation of the magnetic Dirac eigenvalues.
\begin{lemma}\label{lem.lower.dirac}
	Let $\Omg\subset\dR^2$ be a bounded, simply-connected domain with $\mathscr C^2$ boundary. Then for any $B > 0$ we have
	\[
		\lm_1^+(\Omg,B) \ge \lm_1^+(\Omg,0) e^{-2B\varphi^\Omg_{\rm m}}.
	\]
\end{lemma}
\begin{proof}
	Throughout the proof we use the abbreviation $\sH(\Omg) := \sH_0(\Omg)$ for the auxiliary Hilbert space corresponding to the zero vector potential.	First, we observe that the
	bounded and boundedly invertible multiplication operator 
	$\sfM u := e^{-B\varphi^\Omg} u$ acting in $L^2(\Omg)$ is a bijection between $\sH_{\bf A}(\Omg)$ and $\sH_0(\Omg)$. Indeed, for any $u\in\sH_{\bf A}(\Omg)$ we have a decomposition $u = u_0 + u_1$, where $u_0\in H^1(\Omg)$ and $u_1\in L^2(\Omg)$ satisfies $d_{\bf A}^\times u_1 = 0$ and $u_1|_{\p\Omg}\in L^2(\p\Omg)$. By regularity of the torsion function we have $\sfM u_0 \in H^1(\Omg)$.
	Moreover, it holds that $u_1|_{\p\Omg} = \sfM u_1|_{\p\Omg}\in L^2(\p\Omg)$, where we used that torsion function vanishes on the boundary of the domain $\Omg$. Finally, we also notice
	that for any $w\in L^2(\Omg)$ with $d_{\bf A}^\times w\in L^2(\Omg)$ we have
	\begin{equation}\label{eq:identity}
		d_0^\times \sfM w = e^{-B\varphi^\Omg}\Big(-2i \p_{\ov z} w + i B(\p_1\varphi^\Omg) w  - B(\p_2\varphi^\Omg) w\Big) = e^{-B\varphi^\Omg} d_{\bf A}^\times w. 
	\end{equation}
	Thus, we conclude that $\sfM u =\sfM u_0 + \sfM u_1\in\sH(\Omg)$, since $\sfM u_0\in H^1(\Omg)$ and
	by the above identity $\sfM u_1\in\mathscr{H}_{0}^2(\Omg)$. Hence, we have established that $\sfM$ maps $\sH_{{\bf A}}(\Omg)$ into $\sH(\Omg)$. Analogously, one can check that $\sfM^{-1}$ maps $\sH(\Omg)$ into $\sH_{{\bf A}}(\Omg)$. Moreover, since $\sfM$ has no kernel, we conclude that $\sfM$ is a bijection between $\sH_{\bf A}(\Omg)$ and $\sH(\Omg)$.
	Hence, by the variational characterisation~\eqref{eq:minmaxDirac} we get
	\[
	\begin{aligned}
		\lm_1^+(\Omg,B) &= \inf_{u\in \sH_{{\bf A}}(\Omg)\sm\{0\}}
		\frac{\|u\|^2_{L^2(\p\Omg)} + \sqrt{\|u\|^4_{L^2(\p\Omg)}+ 4\|u\|^2_{L^2(\Omg)} \|d_{\bf A}^\times u\|^2_{L^2(\Omg)}}}{2\|u\|^2_{L^2(\Omg)}}\\
		&=\inf_{v\in \sH_0(\Omg)\sm\{0\}}
		\frac{\|v\|^2_{L^2(\p\Omg)} + \sqrt{\|v\|^4_{L^2(\p\Omg)}+ 4\|e^{B\varphi^\Omg}v\|^2_{L^2(\Omg)} \|e^{B\varphi^\Omg}d_0^\times v\|^2_{L^2(\Omg)}}}{2\|e^{B\varphi^\Omg} v\|^2_{L^2(\Omg)}}\\
		&\ge e^{-2B\varphi^\Omg_{\rm m}}\inf_{v\in \sH_0(\Omg)\sm\{0\}}
				\frac{\|v\|^2_{L^2(\p\Omg)} + \sqrt{\|v\|^4_{L^2(\p\Omg)}+ 4\|v\|^2_{L^2(\Omg)} \|d_0^\times v\|^2_{L^2(\Omg)}}}{2\|v\|^2_{L^2(\Omg)}}\\
				& = e^{-2B\varphi^\Omg_{\rm m}}\lm_1^+(\Omg,0),
		\end{aligned}
	\]
	where we used the properties of the mapping $\sfM$, in particular, we employed the identity~\eqref{eq:identity}.
\end{proof}
We also recall a lower bound on $\lm_1^+(\Omg,0)$ for bounded $\mathscr C^2$-smooth simply-connected domains,
\begin{equation}\label{eq:lbnd}
	\lm_1^+(\Omg,0) \ge \sqrt{ \frac{2\pi}{|\Omg|}},
\end{equation}
which is proven in~\cite[Theorem 1]{BFSV17b}, see also~\cite[Theorem 1]{R06}.
\medskip

The proof of Theorem \ref{thm.3} is then identical to the one of Theorem \ref{thm.1}, using the upper bound in Lemma \ref{lem.upper.dirac}, the lower bound in Lemma \ref{lem.lower.dirac}, and the quantitative estimate on the torsion function, Lemma \ref{lem.torsion}.
For the sake of completeness, we provide it below.
\begin{proof}[Proof of Theorem~\ref{thm.3}]
	Let $n\in\dN$ be fixed.	
	The upper bound from Lemma~\ref{lem.upper.dirac} for the disk and the uniform lower bound from Lemma \ref{lem.lower.dirac} applied to $\Omg$ and combined with~\eqref{eq:lbnd} give for all $B \ge M_n$
	\begin{equation}
		\frac{\lambda_n^+(\Omega,B)}{\lambda_n^+(\D,B)} \geq \frac{\sqrt{2\pi} e^{2B (\varphi_{\rm m} ^\D - \varphi_{\rm m}^\Omega)}}{ C_n B^n}.
	\end{equation}
	We then use the quantitative estimate on the torsion function (Lemma \ref{lem.torsion}) to obtain for all $B  \ge M_n$
	\begin{align*}
		\frac{\lambda_n^+(\Omega,B)}{\lambda_n^+(\D,B)} &\geq \frac{\sqrt{2\pi} e^{2c B \alpha(\Omega)^3}}{ C_n B^n}
	= \frac{\sqrt{2\pi}}{C_n} e^{2cB \alpha(\Omega)^3 - n\ln B} \geq \frac{\sqrt{2\pi}}{ C_n} \left(2 c  B \alpha(\Omega)^3 -  n \ln B\right) ,
	\end{align*}
	where we used the inequality $e^x \ge x$ in the last step.
	Thus, the claim of the theorem follows from the above inequality,
	since there exist constants $c_n>0$, $B_n> M_n$ such that
	\[
	 \frac{\sqrt{2\pi}}{C_n} \left(2 c  B \alpha(\Omega)^3 -  n \ln B\right)  \ge 1 + c_n B \alpha(\Omega)^3 - \frac{\ln B}{c_n} 
	\]
	for all $B \ge B_n$.
\end{proof}

\begin{appendix}
	\section{Proof of Lemma~\ref{lem.torsion}}
		For $t \geq 0$ define the level sets $U_t = \lbrace x \in \Omega\colon \varphi^\Omega(x) \geq t \rbrace$, and $\mu(t) = | U_t |$. We start by bounding the perimeter $P(U_t)$ in terms of $\mu(t)$. By Sard's Theorem, for almost all $t$ the boundary $\partial U_t = \lbrace \varphi^\Omega = t \rbrace$ is smooth, and therefore
		\begin{align*}
		P(U_t)^2 = | \lbrace \varphi^\Omega = t \rbrace |^2 = \left( \int_{\lbrace \varphi^\Omega = t \rbrace} \dd \mathcal H^1 \right)^2,
		\end{align*}
		where $\dd \mathcal H^1$ is the measure induced by the Lebesgue measure\footnote{Also called the one-dimensional Hausdorff measure.} on the curve $\partial U_t$.
		By Sard's Theorem again, $\nabla \varphi^\Omega$ is non vanishing on $\lbrace \varphi^\Omega = t \rbrace$ for almost all $t$. Therefore we can use the Cauchy-Schwarz inequality to get
		\begin{align}\label{eq.P2}
		P(U_t)^2 \leq  \left(\int_{\lbrace \varphi^\Omega = t \rbrace} | \nabla \varphi^\Omega |\dd\cH^1\right)\left( \int_{\lbrace \varphi^\Omega = t \rbrace} \frac{1}{|\nabla \varphi^\Omega|}\dd\cH^1\right),
		\end{align}
		for almost all $t\geq 0$. The co-area formula gives
		\[ \mu(t) = \int_{U_t} \dd x = \int_t^{\varphi_{\rm m}^\Omega} \int_{\lbrace \varphi^\Omega = s \rbrace} \frac{1}{|\nabla \varphi^\Omega|}\dd\cH^1 \dd s\]
		and therefore
		\begin{equation}\label{eq.mu'}
		\mu'(t) = -\int_{\lbrace \varphi^\Omega = t \rbrace} \frac{1}{|\nabla \varphi^\Omega|}\dd\cH^1,
		 \end{equation}
		 for almost all $t\geq 0$.
		Moreover, by definition of $\varphi^\Omega$ we have
		\[ \mu(t) = \int_{\lbrace \varphi^\Omega \geq t \rbrace} - \Delta \varphi^\Omega \dd x,\]
		and since the inward-pointing normal to $\partial U_t$ is $\nabla \varphi^\Omega / | \nabla \varphi^\Omega |$, we can use Stokes Theorem to get
		\begin{equation}\label{eq.mu}
		\mu(t) = \int_{\lbrace \varphi^\Omega = t \rbrace} \nabla \varphi^\Omega \cdot \frac{\nabla \varphi^\Omega}{|\nabla \varphi^\Omega|} \dd \mathcal H^1 = \int_{\lbrace \varphi^\Omega = t \rbrace} |\nabla \varphi^\Omega | \dd \mathcal H^1,
		\end{equation}
		for almost all $t\geq 0$.
		Then, we use \eqref{eq.mu'} and \eqref{eq.mu} to rewrite \eqref{eq.P2} as 
		\begin{equation}
		 P(U_t)^2 \leq - \mu(t) \mu'(t),
		\end{equation}
		for almost all $t\geq 0$.
		We now use the quantitative isoperimetric inequality \cite[Theorem 1.1]{FMP} for $U_t$, which gives a universal constant $\mathsf c_{\rm iso} >0$ such that
		\[ P(U_t)^2 \geq 4\pi |U_t| \big( 1 + \mathsf c_{\rm iso} \alpha(U_t)^2 \big),\]
		and we deduce that 
		\[ 4 \pi( 1+  \mathsf{c}_{\rm iso} \alpha(U_t)^2 ) \leq - \mu'(t),\]
		for almost all $t \in (0, \varphi^\Omega_{\rm m})$. Integrating over the whole interval we find
		\[ 4 \pi \varphi^\Omega_{\rm m}  + 4 \pi \mathsf c_{\rm iso} \int_0^{\varphi^\Omega_{\rm m}} \alpha(U_t)^2 \dd t \leq \mu(0) = |\Omega|.\]
		Note that the torsion function of the disk $\cD$ is the radial function $r \mapsto \frac{|\Omega|}{4\pi} - \frac{r^2}{4}$. Therefore,
		\begin{equation}\label{eq.quantlow1}
			\varphi^{\cD}_{\rm m} - \varphi^{\Omega}_{\rm m} \geq \mathsf c_{\rm iso} \int_0^{\varphi_{\rm m}^\Omega} \alpha(U_t)^2 \dd t. 
		\end{equation}
		Now consider the threshold $s_\Omega$ from which $\mu(t)$ gets far from $|\Omega|$, namely
		\begin{equation}
			s_\Omega = \sup \mathscr A, \qquad \mathscr A = \Big\lbrace t \geq 0\colon \mu(t) \geq |\Omega | \Big( 1 - \frac{\alpha(\Omega)}{4} \Big) \Big\rbrace.
		\end{equation}
		For $t \in \mathscr{A}$ we have $ \frac{| \Omega \setminus U_t |}{|\Omega|} \leq \frac 1 4 \alpha(\Omega)$, which implies $\alpha(U_t) \geq \frac{\alpha(\Omega)}{2}$ by standard properties of the Fraenkel asymmetry~\cite[Lemma 2.8]{BP17}. Thus, equation \eqref{eq.quantlow1} gives
		\begin{equation}\label{eq.quantlow2}
			\varphi^{\cD}_{\rm m} - \varphi^{\Omega}_{\rm m} \geq \frac {\mathsf c_{\rm iso}} 4  s_\Omega \alpha(\Omega)^2.
		\end{equation}
		We now consider two complementary cases.
		\begin{itemize}
			\item [\rm (i)] If $\alpha(\Omega) \leq \frac{32 \pi}{|\Omega|} ( \varphi^{\cD}_{\rm m} - \varphi^{\Omega}_{\rm m} )$, then Lemma \ref{lem.torsion} is obvious since $\alpha(\Omega) \geq \frac14\alpha(\Omega)^3$.
			\item [\rm (ii)] If $\alpha(\Omega) \geq \frac{32 \pi}{|\Omega|} ( \varphi^{\cD}_{\rm m} - \varphi^{\Omega}_{\rm m} )$, we use that $\varphi^\Omega - s_\Omega$ is the torsion function on $U_{s_\Omega}$. Therefore, its maximum is smaller than the maximum of the torsion of the disk of volume $|U_{s_\Omega}|$:
			\[ \varphi_{\rm m}^\Omega - s_\Omega \leq \frac{|U_{s_\Omega}|}{4\pi} = \frac{|\Omega|}{4\pi} \Big( 1 - \frac{\alpha(\Omega)}{4} \Big).\]
			Since $\varphi_{\rm m}^{\cD} = \frac{|\Omega|}{4\pi}$, we deduce
			\[ s_\Omega \geq \frac{|\Omega| \alpha(\Omega)}{16\pi} - \big( \varphi^{\cD}_{\rm m} - \varphi^\Omega_{\rm m} \big) \geq \frac{|\Omega| \alpha(\Omega)}{32\pi}. \]
			Then \eqref{eq.quantlow2} gives $\varphi_{\rm m}^{\cD} - \varphi_{\rm m}^\Omega \geq \frac{\mathsf c_{\rm iso}}{128 \pi} |\Omega| \alpha(\Omega)^3$, which concludes the proof.
		\end{itemize}
		Finally, we remark that thanks to the bound $\mathsf c_{\rm iso} \ge \frac{1}{8}$ given in~\cite[Theorem 1]{LZDJ15} we can state that inequality in an explicit form
		\[
		\varphi_{\rm m}^{\cD} - \varphi_{\rm m}^\Omega \geq \frac{1}{1024 \pi} |\Omega| \alpha(\Omega)^3
	\]
\end{appendix}

\section*{Acknowledgements}

The authors thank D. Krej{\v{c}}i{\v{r}}{\'{\i}}k and N. Raymond for valuable comments. L.M. is grateful to L. Junge for preliminary discussions on the topic and sharing the reference \cite{baur}.
The work of L.M. is supported by the European Union (via the ERC Advanced Grant MathBEC - 101095820). Views and opinions expressed are however those of the authors only and do not necessarily reflect those of the European Union or the European Research Council.

\bibliographystyle{plain}

\end{document}